\documentclass[3p,times]{elsarticle}
\usepackage{filecontents}
\usepackage{lipsum}
\usepackage{sans}

\usepackage{lineno,hyperref}
\modulolinenumbers[5]

\usepackage{multicol}
\usepackage{lipsum}
\usepackage{mwe}

\usepackage{graphics}
\usepackage{graphicx}
\usepackage{epsfig}
\usepackage{amssymb}
\usepackage{amsmath}
\usepackage{lineno}
\usepackage{enumerate}
\usepackage{indentfirst}
\usepackage{graphicx}
\usepackage{graphics}
\usepackage{latexsym}
\usepackage{amssymb}
\usepackage{verbatim}
\usepackage{amsthm}
\usepackage{amsfonts}
\usepackage[para,online,flushleft]{threeparttable}
\usepackage{multirow}
\usepackage{floatrow}
\usepackage{amsfonts}
\usepackage{booktabs}
\usepackage{siunitx}
\usepackage{tabularx}
\usepackage{rotating}
\usepackage{array}
\usepackage{booktabs}
\usepackage{ltablex,booktabs}
\usepackage{lineno}
\usepackage{caption}
\usepackage{graphicx}
\usepackage{graphics}
\usepackage{latexsym}
\usepackage{amssymb}
\usepackage{verbatim}
\usepackage{amsmath}
\usepackage{amsthm}
\usepackage{amsfonts}
\usepackage[utf8]{inputenc}
\usepackage[english]{babel}
\usepackage[usenames, dvipsnames]{color}
\usepackage{rotating}

\usepackage{booktabs,caption}
\usepackage[flushleft]{threeparttable}

\usepackage{amssymb}
\usepackage{pifont}
\usepackage{graphicx}
\usepackage{subcaption}

\usepackage{xcolor,colortbl}

\definecolor{Gray}{gray}{0.85}
\definecolor{LightCyan}{rgb}{0.88,1,1}
\usepackage{array, boldline, makecell, booktabs}

\usepackage{xcolor}
\usepackage{mdframed}
\usepackage{titletoc}
\usepackage{etoolbox}
\usepackage{array,multirow}
\usepackage{soul}

\newtheorem{theorem}{Theorem}[section]
\newtheorem{lemma}[theorem]{Lemma}
\newtheorem{proposition}[theorem]{Proposition}

\theoremstyle{definition}
\newtheorem{definition}[theorem]{Definition}
\newtheorem{example}[theorem]{Example}

\theoremstyle{remark}
\newtheorem{remark}[theorem]{Remark}
\numberwithin{equation}{section}

\biboptions{comma,square}
\biboptions{}
\usepackage{booktabs}
\usepackage{colortbl}
\usepackage{xcolor}
\usepackage{xfrac}
\usepackage{caption}
\usepackage{xcolor}
\usepackage{listings}

\usepackage{lipsum}

\usepackage{mathpazo}
\usepackage{caption}
\usepackage[english]{babel}
\usepackage[utf8]{inputenc}
\usepackage{tabulary}
\usepackage{colortbl}
\usepackage[utf8]{inputenc}
\usepackage[english]{babel}

\usepackage{multicol}
\usepackage{colortbl}

\usepackage{mathpazo}
\usepackage{booktabs,caption,array,ragged2e}
\newcolumntype{C}[1]{>{\centering\arraybackslash}p{#1}}
\newcolumntype{L}[1]{>{\RaggedRight\arraybackslash}p{#1}}

\usepackage{mathpazo} 

\usepackage{booktabs,array,ragged2e,dcolumn}
\usepackage{array}
\makeatletter
\newcommand{\thickhline}{%
	\noalign {\ifnum 0=`}\fi \hrule height 1pt
	\futurelet \reserved@a \@xhline
}
\newcolumntype{"}{@{\hskip\tabcolsep\vrule width 1pt\hskip\tabcolsep}}
\makeatother
\usepackage{color, colortbl}

\usepackage[first=0,last=9]{lcg}

\usepackage[T1]{fontenc}
\usepackage{makecell,tabularx}

\usepackage{graphicx}

\usepackage{float}

\usepackage{subcaption}

\newcommand\blfootnote[1]{%
	\begingroup
	\renewcommand\thefootnote{}\footnote{#1}%
	\addtocounter{footnote}{-1}%
	\endgroup
}

\hypersetup{hidelinks=true}

\makeatletter
\def\ps@pprintTitle{%
	\let\@oddhead\@empty
	\let\@evenhead\@empty
	\def\@oddfoot{\reset@font\hfil\thepage\hfil}
	\let\@evenfoot\@oddfoot
}
\makeatother

\begin{document}

\begin{frontmatter}

\title{Goppa code and quantum stabilizer codes from plane  curves given by separated polynomials}

\author{Vahid Nourozi$^{a,*}$ and Farzaneh Ghanbari$^{b}$}



\address[mysecondaryaddress]{The Klipsch School of Electrical and Computer Engineering, New Mexico State University,
Las Cruces, NM 88003 USA}
\address[mymainaddress]{Department of Pure Mathematics, Faculty of Mathematical Sciences,
 Tarbiat Modares University, P.O.Box:14115-134, Tehran, Iran.}


\blfootnote{Email address: nourozi@nmsu.edu}
\cortext[mycorrespondingauthor] {Corresponding author}

\begin{abstract}
In this paper, we examine algebraic-geometric (AG) codes associated with curves generated by separated polynomials, and we create AG codes and quantum stabilizer codes from these curves by varying their parameters. Our research involves a thorough examination of the curves' algebraic features as well as the creation of Goppa codes over them. Extending these findings, we create quantum stabilizer codes, revealing that quantum codes built from Hermitian self-orthogonal AG codes have acceptable parameters, improving the reliability and performance of communication networks.
\end{abstract}

\begin{keyword}
\texttt{Goppa code, Finite fields, algebraic geometry codes, quantum stabilizer codes, plane curves given by separated polynomials.}
\end{keyword}

\end{frontmatter}

\section{Introduction}
After Goppa’s construction \cite{77tores}, ideas from algebraic geometry proved influential in coding theory. He proposed the brilliant concept of connecting code $C$ to a (projective, geometrically irreducible, non-singular, algebraic) curve $\mathcal{X}$ defined over $\mathbb{F}_q$, a finite field with $q$ elements. This code comprises two divisors, $D$ and $G$, with one of them, $D$, being the sum of $n$ distinct $\mathbb{F}_q$ rational points of $\mathcal{X}$. It becomes out that the minimum distance $d$ of $C$ satisfies
$$d \geq n - \mbox{deg}(G).$$
This is one of the key aspects of Goppa’s work. In general, there is no known lower bound for the minimum distance of an arbitrary code. This bound is only relevant if $n$ is sufficiently large. Since $n$ is upper bounded by the Hasse-Weil upper bound

$$1 + q + 2g\sqrt{q},$$
where $g$ is the genus of the underlying curve; it is of significant interest to study curves with numerous rational points; see \cite{66tores} and \cite{33tores}.

$AG$ Codes from the Hermitian curve have been extended in numerous studies; see \cite{10maria,24maria,25maria,26maria,44maria,46maria,47maria}.
In addition, a family of Hermitian self-orthogonal classical codes are derived from the algebraic geometry codes developed in
\cite{44jin,55jin,66jin}. Also, Vahid introduced the Goppa code from Hyperelliptic Curve \cite{aut, N23, miss, auta, shiraz}, from plane curves given by separated polynomials \cite{ M23, behsep, Nourozi2024, code, esfahan}, and he explained them in his Ph.D. dissertation in \cite{phd}. Optimization frameworks are instrumental in addressing complex challenges across disciplines, including power systems and quantum coding theory. Refs \cite{hamid, hamid1} utilize mixed-integer programming to explore trade-offs in resource allocation, emphasizing the balance between operational efficiency and cost in ancillary service markets. Similarly, Ref \cite{hamid2} introduces robust optimization techniques to address reserve deliverability under uncertainty, showcasing innovative methods to simplify computational complexity while preserving system reliability. These works demonstrate how optimization-based approaches manage trade-offs between performance metrics and constraints, a concept central to power systems and robust quantum systems design.

Section \ref{se33} introduces basic notions and preliminary results of $AG$ codes and plane curves given by separated polynomial curves. Section \ref{sec44} presents the quantum Goppa code on curve $\mathcal{X}$.

\section{Preliminary}\label{s2}
\subsection{Curves given by separated polynomials}
In this paper $\mathcal{X}$ is a plane curve defined over the algebraic closure $K$ of a prime finite field $\mathbb{F}_q$ by an equation
\begin{equation*}
A(Y)=B(X),
\end{equation*}
satisfies the following conditions:
\begin{itemize}
\item[(1)] $\mbox{deg}(\mathcal{X})\geq 4$;
\item[(2)] $A(Y)= a_nY^{p^n}+a_{n-1}Y^{p^n-1}+\cdots +a_0Y, \hspace{0.3cm} a_j \in K, \hspace{0.3cm} a_0,a_n \neq 0$;
\item[(3)] $B(X)=b_mX^m+b_{m-1}X^{m-1}+\cdots + b_1X + b_0, \hspace{0.3cm} b_j \in K,\hspace{0.3cm}b_m \neq 0$;
\item[(4)] $m \not\equiv 0 \hspace{0.3cm} (\mbox{mod} p)$;
\item[(5)] $n\geq 1, \hspace{0.3cm} m\geq 2$.
\end{itemize}
Note that $(2)$ occurs if and only if $A(Y +a) = A(Y )+A(a)$ for every $a \in K$, that is, the polynomial $A(Y)$ is additive. The basic properties of $\mathcal{X}$ are presented in the following lemmas; see [\cite{tores}, Section 12.1].

\begin{lemma}
Curve $\mathcal{X}$ is an irreducible plane curve with a maximum of one singular point.
\begin{itemize}
\item[(i)] If $\mid m - p^n \mid =1$, then $\mathcal{X}$ is non-singular.
\item[(ii)] $\mathcal{X}$ has genus $g=\frac{(p^n-1)(m-1)}{2}$.
\end{itemize}
\end{lemma}

We suppose that $q=p^n$ and $\mathcal{X}$ is a curve over $\mathbb{F}_q$ with the above condition defined by following an equation

$$y^q+y=x^m,$$
where $q-m=1$.

\subsection{Algebraic Geometry Codes}
In this paper we let $\mathbb{F}_q(\mathcal{X})$ (resp. $\mbox{Div}_q(\mathcal{X})$) denotes the field of the $\mathbb{F}_q$-rational functions (resp. $\mathbb{F}_q$ of divisors) for $\mathcal{X}$. If $f \in \mathbb{F}_q(\mathcal{X})\setminus \{0\}$, $ \mbox{div}(f)$ denotes the divisor associated with $f$. For $A \in \mbox{Div}_q(\mathcal{X})$, $\mathcal{L}(A)$ denotes the Riemann-Roch $\mathbb{F}_q$-vector space associated with $A$, i.e.,

\begin{equation*}
\mathcal{L}(A) = \{ f \in \mathbb{F}_q(\mathcal{X})\setminus \{0\}: A + \mbox{div}(f) \succeq 0\} \cup \{0\}.
\end{equation*}
and its dimension over $\mathbb{F}_q$ is denoted by $\ell(A)$.

Let $P_1, \cdots , P_n$ be pairwise distinct $K$-rational points of $\mathcal{X}$, and $D = P_1 +\cdots+ P_n$ of degree $1$. Choose divisor $G$ on $\mathcal{X}$ such that $\mbox{supp}(G) \cap \mbox{supp}(D) = \phi$.
\begin{definition}
The algebraic geometry code (or $AG$ code) $C_{\mathcal{L}}(D,G)$ is associated with $D$ and $G$’s divisors, defined as
\begin{equation*}
C_{\mathcal{L}}(D,G) := \{(x(P_1), \cdots, x(P_n)) \mid x \in \mathcal{L}(G)\} \subseteq \mathbb{F}^n_q.
\end{equation*}
\end{definition}

The minimum distance $d$
satisfies $d \geq d^{\star} = n - \mbox{deg}(G)$, where $d^{\star}$ is called the Goppa designed minimum distance
of $C_{\mathcal{L}}(D,G)$, if $\mbox{deg}(G) > 2g - 2$, then by the Riemann-Roch Theorem $k =
\mbox{deg}(G) - g + 1$; see [\cite{2121}, Th. 2.65]. The dual code $C^{\perp}(D,G)$ is an $AG$ code with dimensions
$k^{\perp} = n - k$ and minimum distance $d^{\perp} \geq \mbox{deg}G - 2g + 2$.  Let $H(P)$ be
the Weierstrass semigroup associated with $P$, that is
\begin{equation*}
H(P) := \{n \in \mathbb{N}_0 \mid \exists f \in \mathbb{F}_q(\mathcal{X}), \mbox{div}_{\infty}(
f) = nP\} = \{ \rho_0 = 0 < \rho_1 < \rho_2 < \cdots \}.
\end{equation*}

Recall that the Hermitian inner product for two vectors $a = (a_1, \cdots , a_n),$ $b = (b_1, \cdots , b_n)$ in $\mathbb{F}_q^n$, is defined by $<a, b>_H:=\sum_{i=1}^n a_ib_i^q$. For a linear code $C$ over $\mathbb{F}_q^n$, the Hermitian dual of $C$ is determined by
\begin{equation*}
C^{\perp H} := \{ v \in \mathbb{F}_q^n: <v,c>_H=0 \hspace{0.3cm} \forall c \in C\}.
\end{equation*}
Thus,$C$ is a Hermitian self-orthogonal if $C \subseteq C^{\perp H}.$
\section{Goppa Code Over Curve $\mathcal{X}$}\label{se33}
Let $r \in \mathbb{N}$, with the notation of Section \ref{s2}, we consider the sets
\begin{equation*}
\mathcal{G}:= \mathcal{X}(\mathbb{F}_q), \hspace{0.6cm} \mathcal{D}:=\mathcal{X}(\mathbb{F}_{q^2})\setminus \mathcal{G}
\end{equation*}
$\mathcal{G}$ is the intersection of $\mathcal{X}$ with the plane $t = 0$. Fix the $\mathbb{F}_{q^2}$ divisors
\begin{equation*}
G:= \sum_{P \in \mathcal{G}} rP \hspace{0.5cm} \mbox{and} \hspace{0.5cm} D:= \sum_{P \in \mathcal{D}}P,
\end{equation*}
where $\mbox{deg}(G) = r(q^2 - q + 1)$ and $\mbox{deg}(D) = 2q^3 - 2q^2 + 2q + 1$. Let $C$ be the $C_{\mathcal{L}}(D,G)$ Algebraic Geomtery code over $\mathbb{F}_{q^2}$. With length $n=q^3$, the minimum distance $d$ and dimension $k$. The designed minimum distance of $C$ is
\begin{equation*}
d^{*}= n - \mbox{deg}(G) = q^3 - r(q^2 - q +1).
\end{equation*}
From H. Stichtenoth \cite{stich} makes the following remark:
\begin{remark}\label{2.9}
Let $\mathcal{X}$ be a curve in this paper, and $D$ and $G$ be the divisors as above. Then
$$C^{\perp}(D,G) = C(D,D-G+K),$$
where $K=div(\eta) \in Div_q(\mathcal{X})$ is a canonical divisor defined by a differential $\eta$, such that $\nu_{P_i}(\eta)=-1$ and $\mbox{res}_{P_i}(\eta)=1$ for each $i= 1 , 2, \cdots , n$.
\end{remark}
\begin{lemma}\label{2.11}
The basis of $\mathcal{L}(G)$, for $r\geq 0$, is given by

$$\{x^iy^j \mid iq + j(q-1) \leq r , i \geq 0 , 0 \leq j \leq q-1 \}.$$
\end{lemma}
\begin{proof}
We know that $(x)_{\infty} = q P_{\infty}$ and $(y)_{\infty} = (q-1)P_{\infty}$; therefore, the set above is contained in $\mathcal{L}(G)$. Additionally, the restriction on $j$ is linearly independent of $\mathbb{F}_{q^2}$. The Weierstrass semigroup $H(P_{\infty})$ is generated by $q$ and $q - 1$ at $P_{\infty}$. Suppose that $\mathcal{L}(G) = \mathcal{L}(\rho_{\ell}P_{\infty})$ where $\rho_{\ell} \leq r \leq \rho_{\ell +1}$ and $H(P_{\infty}) = \{ \rho_0 = 0 < \rho_1< \cdots \}.$ Then
\begin{equation}
\mbox{dim}_{\mathbb{F}_q}(\mathcal{L}(G)) = \sharp \{iq + j(q-1)  \leq r , i \geq 0 , 0 \leq j \leq q-1 \}.
\end{equation}
\end{proof}

In the following lemma, consider the codes
$C_r := C_{\mathcal{L}}(D,G)$
and
$k_r:= \mbox(dim)_{\mathbb{F}_{q^2}}(C_r)$. Also, we indicate the divisor $div(x)$ with $(x)$.

\begin{lemma}\label{2.12}
We have
$$C_r^{\bot} = C_{q^3 + q^2 -3q -r}.$$
Hence, $C_r$ is self-orthogonal when $2r \leq q^3 + q^2 - 3q$.
\end{lemma}

\begin{proof}
We have $ C_r^{\bot} = C(D,D-G+W)$, where $W$ is the canonical divisor in Remark \ref{2.9}. In this case, we calculated $\eta$. Let $t:= x^{q^2} - x = \prod_{a \in \mathbb{F}_{q^2}}(x-a)$, and $\eta := dt/t$. Therefore
$$(x-a) = \sum_{b^q+b=a^{q-1}} P_{a,b}-qP_{\infty}$$
thus,
$$(t)=D=q^3 P_{\infty}.$$
Additionally, $(dt)=(dx)=(2g-2)P_{\infty}= (q^2 -3q)P_{\infty}$. So that,
\begin{equation*}
\nu_P(\eta) = -1 \hspace{0.6cm} \mbox{and} \hspace{0.6cm} \mbox{res}_P \eta =1 \hspace{0.6cm} \mbox{for} \hspace{0.6cm} P\in \mbox{Supp}(D).
\end{equation*}
Because $D-G-(\eta) = D-G-D+q^3P_{\infty}+(q^2-3q)P_{\infty}= (q^3+q^2-3q-r)P_{\infty}$, the statement follows.
\end{proof}

We suppose that $T(r):= \sharp \{iq + j(q-1)  \leq r , i \geq 0 , 0 \leq j \leq q-1 \}$.
\begin{proposition}
\begin{itemize}
\item[(1)]
If $r<0$, then $k_r =0$;
\item[(2)]
If $ 0 \leq r \leq q^2 -3q$ then $k_r = T(r)$;
\item[(3)]
If $q^2 - 3q < r < q^3$, then $k_r = r(q^2-q+1) - \frac{(q-1)(q-2)}{2}$;
\item[(4)]
If $q^3 \leq r \leq q^3+q^2-3q$, then $k_r = q^3 - T(q^3-q^2-3q-r)$;
\item[(5)]
If $r > q^3- q^2 -3q$, then $k_r=q^3$.
\end{itemize}
\end{proposition}

\begin{proof}
\begin{itemize}
\item[(1)] If $r<0$, it is trivial that $k_r=0$.
\item[(2)]
If $ 0 \leq r \leq q^2 -3q$, then follows from lemma \ref{2.11}.
\item[(3)]
According to the Riemann-Roch theorem $k_r = \mbox{deg}(G) + 1 - g$, as $n > \mbox{deg}( G) > 2g - 2.$
\item[(4)]
Let $r^{\prime} := q^3+q^2-3q-r$, satisfies $0 \leq r \leq q^2-3q$. Then from Lemma \ref{2.12}, $k_r=q^3 - \mbox{dim}_{\mathbb{F}^{q^2}}(C_{r^{\prime}})$ and the statement
follows.
\item[(5)] If $r > q^3- q^2 -3q$, then $C^{\perp}_r=\{0\}$,  and thus $\mbox(dim)_{\mathbb{F}_{q^2}}(C_r) = n = k_r$.
\end{itemize}

\end{proof}

\begin{proposition}
$C$ is monomially equivalent to one-point code $C(D, r(q^2 -q + 1)P_{\infty})$.
\end{proposition}
\begin{proof}
If $G^{\prime} = r(q^2  -q + 1)P_{\infty}$,  then $G = G^{\prime} + (t^r)$. This claim follows the statement in the proof of [Proposition 3.2. \cite{mariaree}].
\end{proof}

\begin{theorem}\label{3.4}
For $r \leq q^2 +q -3$, $C_r$ is a Hermitian self-orthogonal.
\end{theorem}
\begin{proof}
If $r \leq q^2 +q -3$, then we have $rq\leq q^3 +q^2 -3q -2 -r$. Hence, the result
follows from the Lemma \ref{2.12}.
\end{proof}
\section{Quantum Stabilizer Code Over Curve $\mathcal{X}$}\label{sec44}

In this section, we use the Hermitian self-orthogonality of $C_r$ produced in the previous section to construct and analyze quantum stabilizer codes.

We require the following lemma as a result of quantum codes obtained from Hermitian self-orthogonal
classical codes.

\begin{lemma}\cite{ashi}\label{4.1}
There is a $q$-ary $[[n, n - 2k, d^{\perp}]]$ quantum code whenever there exists a $q$-ary classical
Hermitian self-orthogonal $[n, k]$ linear code with dual distance $d^{\perp}$.
\end{lemma}
We can derive our main result using Lemma \ref{4.1}, of quantum codes with classical Hermitian self-orthogonal codes. We Then provide several examples that prove that the quantum codes created from our theorem are indeed good.

\begin{theorem}\label{qcode}
Let $q$ be a power of $s\geq 1$, then for the curve $\mathcal{X}$, there is a $q$-ary $[[q^3, q^3 + q^2 -3q-2r, r +2q -q^2]]_q$ quantum code for any positive value of $q^2 -2 \leq r \leq q^2 +q -3$.
\end{theorem}
\begin{proof}
The proof directly follows from Theorems \ref{3.4} and Lemma \ref{4.1}.
\end{proof}

\begin{example}
For $q = 3$ and $7 \leq r  \leq 9$, Theorem \ref{qcode} produces 3-ary $[[ 27,27-2r,r-3]]_3$ quantum codes. Quantum codes can drive that, $[[27,13,4]]_3$, $ [[27,11,5]]_3$, and $[[27,9,6]]_3$. These codes have good parameters. For instance, the quantum code $[[27, 13, 6]]_3$ is presented in Table \cite{online}. This example implies that our quantum code has a shorter distance for the same dimensions and lengths.
\end{example}

\begin{example}
For $q = 5$ and $18 \leq r  \leq 22$, Theorem \ref{qcode} produces 5-ary $[[ 125,135-2r,r-15]]_5$ quantum codes. Quantum codes can drive that, $[[125,99,3]]_5 , [[125,97,4]]_5 , [[125,95,5]]_5, [[125,93,6]]_5$ and $[[125,91,7]]_5$. Where these codes have good parameters. For instance, the quantum codes $[[125,99,10]]_5$ and $[[125,93,12]]_5$ are
presented in Table \cite{online}. As a result of this example, our quantum code has a smaller distance for the same dimension and length.
\end{example}

\section*{Conclusion}
In this work, we investigated the Goppa codes of the plane curve given by separated polynomial curves.
We introduced and analyzed the Quantum Stabilizer Code over the plane curve given by separated polynomials. For example, we say that the Quantum Stabilizer Code improves with good parameters over this
curve.

\paragraph*{\textbf{Acknowledgements.}}
This paper was written while Vahid Nourozi was visiting Unicamp (Universidade Estadual de Campinas) supported by TWAS/Cnpq (Brazil) with fellowship number $314966/2018-8$.



\bibliographystyle{elsevier}

 \end{document}